\documentclass[10pt, a4paper, oneside, reqno]{amsart}
\usepackage[margin=1in]{geometry}
\usepackage{amsmath,amssymb,amsthm}
\usepackage{mathrsfs}
\usepackage{amsfonts}
\usepackage{hyperref}
\usepackage{enumitem}
\usepackage{graphicx}
\usepackage{mathrsfs}
\usepackage{upgreek}
\usepackage{mathtools}
\usepackage{tikz}       
\usepackage{amsmath}    
\theoremstyle{plain}
\newtheorem{theorem}{Theorem}[section]
\newtheorem{proposition}{Proposition}[section]

\theoremstyle{definition}
\newtheorem{definition}{Definition}[section]

\theoremstyle{remark}
\newtheorem{remark}{Remark}[section]

\title{Minimal networks on $S^2$}
\author{Xuyan Liu}
\address{College of Mathematics and Statistics, Chongqing University, \\
	Chongqing 401331, P. R. China}
\email{202306021002@cqu.edu.cn}
\date{\today}
\email{$^1$ 202306021002@cqu.edu.cn }  

\date{\today}

\begin{document}

\maketitle

\begin{abstract}
This paper focuses on minimal networks in geometric measure theory and the calculus of variations. We first review the basic theory of minimal networks in Euclidean space, including the definition of immersed networks, 1-rectifiable currents, and the calibration method. Then we locally extend the classical theory and calibration techniques to the standard unit sphere \(S^2\). By redefining adapted \(\mathbb{R}^2\)-valued co-vectors, differential forms, currents, and calibrations on \(S^2\), we establish a local framework for minimal networks on the sphere. We prove that a spherical minimal network composed of great circle arcs with \(120^\circ\) triple junctions is locally length-minimizing within sufficiently small geodesic balls on \(S^2\). The results partially fill the theoretical gap for minimal networks in constant curvature spaces and provide a foundation for further investigations on minimal networks in higher-dimensional Riemannian manifolds.
\end{abstract}

\section{Introduction}
The minimal network problem is a classical topic in geometric measure theory and the calculus of variations, which aims to find networks of minimal length connecting given points. Most classical results are established in the Euclidean plane, while a complete theory for constant-curvature Riemannian manifolds remains to be developed.

In this paper, we locally extend the theory of minimal networks and the calibration method from the Euclidean plane to the standard unit sphere \(S^2\). We redefine \(\mathbb{R}^2\)-valued co-vectors, differential forms, currents, and calibrations adapted to spherical geometry. Using exponential maps and local metric perturbation estimates, we prove that spherical minimal networks composed of great-circle arcs with \(120^\circ\) triple junctions are \textbf{locally length-minimizing only within sufficiently small geodesic balls} on \(S^2\), without obtaining global minimality results.

Our work  partially enriches the theory of minimal networks on constant-curvature spaces, and provides a theoretical reference and technical basis for future research on extending such results to higher-dimensional Riemannian manifolds and more general surfaces.
\quad

\section{Preliminaries}
\subsection{Basic concepts}
\begin{definition}[\cite{r7}]
	Let $N\in\mathbb{N}^*$ be a positive integer. For any index $i\in\{1,2,\dots,N\}$, define an elementary edge $E_i:=[0,1]\times\{i\}$. The total edge set is the union of all elementary edges, i.e.,
	\[
	E:=\bigcup_{i=1}^N E_i.
	\]
	The vertex set consists of all endpoints of the elementary edges, given by
	\[
	V:=\bigcup_{i=1}^N \{(0,i),(1,i)\}.
	\]
	Let $\sim$ be an equivalence relation on the vertex set $V$ satisfying the distinguishability condition, meaning any two distinct points in $V$ can be distinguished. The quotient space induced by the equivalence relation is denoted by $G$, i.e.,
	\[
	G:=E/\sim.
	\]
\end{definition}

\begin{definition}[\cite{r7}]
	A pair $\mathcal{N}=(G,\Gamma)$ is called an \textbf{immersed network}, where $\Gamma:G\to\mathbb{R}^2$ is a continuous map and $G$ is a connected graph as defined above. For any elementary edge $E_i$, the restriction $\gamma^i:=\Gamma|_{E_i}$ is a $C^1$-immersion (including immersion at boundary points).
\end{definition}

\begin{definition}[\cite{r7}]
	An immersed network $\mathcal{N}=(G,\Gamma)$ is called an \textbf{immersed triple-junction network} if every branch point (singular point) of the graph $G$ is of order 3.
\end{definition}
\begin{figure}[htbp]
	\centering
	\begin{tikzpicture}[scale=1.0]
		
		\draw[thick] (0,0) -- (1.7,0);
		\draw[thick] (1.7,0) -- (2.7,1.0);
		\draw[thick] (1.7,0) -- (2.7,-1.0);
		
		\fill (0,0) circle (2pt) node[below] {endpoint};
		\fill (1.7,0) circle (3pt) node[below left] {$m$};
		\fill (2.7,1.0) circle (2pt) node[above] {endpoint};
		\fill (2.7,-1.0) circle (2pt) node[below] {endpoint};
		\node at (1.3,-1.7) {Abstract graph $G$};
		
		\draw[thick,blue] (5,0) .. controls (6,0.3) .. (7,0.3);
		\draw[thick,blue] (7,0.3) .. controls (8,1.0) .. (8.8,0.9);
		\draw[thick,blue] (7,0.3) .. controls (8,-0.6) .. (8.8,-0.7);
		
		\fill (5,0) circle (2pt) node[below left] {$\Gamma(\text{endpoint})$};
		\fill (7,0.3) circle (3pt) node[right=8pt] {$\Gamma(m)$};
		\fill (8.8,0.9) circle (2pt) node[above right] {$\Gamma(\text{endpoint})$};
		\fill (8.8,-0.7) circle (2pt) node[below right] {$\Gamma(\text{endpoint})$};
		
		\node at (6, 0.7) {$\gamma^1$};
		\node at (8.2, 1.2) {$\gamma^2$};
		\node at (8.2, -1.0) {$\gamma^3$};
		
		\node at (6.9,-1.7) {Immersed triple-junction network $\mathcal{N}=(G,\Gamma)$};
		
		\draw[->,dashed,gray] (1.7,0) -- (7,0.3);
		\draw[->,dashed,gray] (0,0) -- (5,0);
		\draw[->,dashed,gray] (2.7,1.0) -- (8.8,0.9);
		\draw[->,dashed,gray] (2.7,-1.0) -- (8.8,-0.7);
		
	\end{tikzpicture}
	\caption{Sketch of an immersed triple-junction network}
	\label{fig:triple_junction_network}
\end{figure}
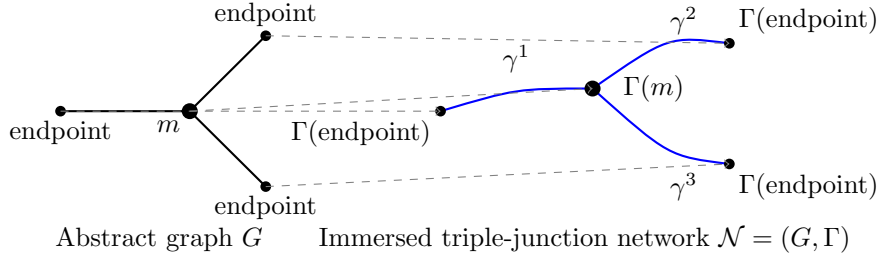

\begin{remark}
	On the left is the abstract topological graph $G$ with exactly one vertex of order 3.
	On the right is the geometric network immersed into the plane $\mathbb{R}^2$ via a continuous immersion $\Gamma$,
	where all junctions are of order 3, satisfying the definition of an immersed triple-junction network.
\end{remark}

\begin{definition}[\cite{r7}]
	An immersed network $\mathcal{N}=(G,\Gamma)$ is called a \textbf{minimal network} if it satisfies all of the following conditions:
	\begin{enumerate}[label=\textbf{(\alph*)},leftmargin=*]
		\item For each edge $e_i\in E(G)$, the restricted curve $\gamma^i:=\Gamma|_{e_i}$ is embedded as a line segment in $\mathbb{R}^2$;
		\item The interiors of any two distinct edge curves $\gamma^i$ and $\gamma^j$ ($i\neq j$) do not intersect or overlap;
		\item For any index $i$, $\pi(0,i)\neq\pi(1,i)$, i.e., the two endpoints of a single elementary edge do not coincide;
		\item All branch points (singular points) of the graph $G$ are of order 3;
		\item At every branch point, the sum of the inward tangent vectors of all adjacent edges is zero.
	\end{enumerate}
\end{definition}

\subsection{Properties}
\begin{theorem}[\cite{r7}]
	Let $\Gamma_*: G \to \mathbb{R}^2$ be a minimal network on a set of endpoints $\{p_1, \dots, p_n\} \subset \mathbb{R}^2$. Then there exists a 1-rectifiable current $\hat{T}$ with coefficients in $\mathcal{G}$ satisfying:
	\begin{enumerate}[label=(\roman*)]
		\item Boundary condition: $\partial \hat{T} = B = \sum_{i=1}^n c_i \delta_{p_i}$, where $c_i \in \mathcal{G}$;
		\item Support: $\operatorname{supp}(\hat{T}) = \Gamma_*(G)$;
		\item Mass equivalence: $L(\Gamma_*) = M(\hat{T})$, where $L(\cdot)$ denotes the length of the network and $M(\cdot)$ denotes the mass of the current;
		\item Existence of calibration: There exists $\omega \in C_c^\infty(\mathbb{R}^2, M^{2\times 2}(\mathbb{R}))$ that is a calibration for $\hat{T}$;
		\item Minimality: $\hat{T}$ is mass-minimizing among all 1-normal currents with coefficients in $\mathbb{R}^2$ and boundary $B$.
	\end{enumerate}
\end{theorem}

\begin{proof}
	\textbf{Step 1: Construction of the current}
	Translate one endpoint of $\Gamma_*$ to the origin of $\mathbb{R}^2$ and apply a rotation such that each edge of $\Gamma_*$ is parallel to one of $g_1, g_2, g_3$. Define the 1-rectifiable current $\hat{T} = [\Sigma, \tau, \theta]$, where:
	the rectifiable set $\Sigma = \Gamma_*(G)$;
	the orientation $\tau: \Sigma \to \mathbb{S}^1$ (unit circle) is constant on the interior of each edge, and the multiplicity function $\theta: \Sigma \to \mathcal{G}$ is also constant on the interior of each edge.
	
	If $x \in \Sigma$ is an interior point lying on an edge parallel to $g_i$, then $\tau(x) = g_i$ and $\theta(x) = g_i$. The values of $\tau$ and $\theta$ at endpoints and branch points can be defined arbitrarily, as they do not affect boundary or mass calculations.
	
	\textbf{Step 2: Verification of boundary and support}
	The boundary contribution of the current at branch points is zero. Hence,
	\[
	\begin{aligned}
		\partial \hat{T} &= \sum_{i=1}^n c_i \delta_{p_i} \\
		&= B,
	\end{aligned}
	\]
	where $c_i \in \mathcal{G}$ is determined by the orientation and multiplicity at the endpoints. By the definition of support (the set where $\theta \neq 0$), it is clear that
	\[
	\operatorname{supp}(\hat{T}) = \Gamma_*(G).
	\]
	
	\textbf{Step 3: Equivalence of mass and length}
	The length of $\Gamma_*$ decomposes into the sum of the lengths of its edges:
	\[
	L(\Gamma_*) = \sum_{e \in E(G)} \mathcal{H}^1(\Gamma_*(e)),
	\]
	where $E(G)$ is the edge set of the graph $G$. The mass of the 1-rectifiable current is defined as
	\[
	M(\hat{T}) = \int_{\Sigma} \|\theta(x)\| d\mathcal{H}^1(x).
	\]
	By construction, for each edge $e \in E(G)$, $\theta|_{\Gamma_*(e)} = g_i$ and $\|g_i\| = 1$. Therefore,
	\[
	\begin{aligned}
		M(\hat{T}) &= \sum_{e \in E(G)} \int_{\Gamma_*(e)} 1 \, d\mathcal{H}^1 \\
		&= \sum_{e \in E(G)} \mathcal{H}^1(\Gamma_*(e)) \\
		&= L(\Gamma_*).
	\end{aligned}
	\]
	
	\textbf{Step 4: Verification of calibration}
	Take the constant matrix $\omega = \operatorname{Id} = \begin{pmatrix} 1 & 0 \\ 0 & 1 \end{pmatrix}$. We verify that it satisfies the three core conditions of a calibration (Definition 2.41):
	
	1. Closure: $\omega$ is a constant matrix, so its exterior derivative $d\omega = 0$;
	
	2. Bounded operator norm: The unit ball of the hexagonal norm is convex with extreme points in $\mathcal{G}$. For any unit vector $\nu = (\cos\alpha, \sin\alpha)$, compute the dual norm:
	\[
	\begin{aligned}
		\|\omega(\nu)\|_* &= \sup_{g \in \mathcal{G}, \|g\|=1} |\langle \omega(\nu), g \rangle| \\
		&\leq 1.
	\end{aligned}
	\]
	Taking $g_1 = (1,0)$ as an example:
	\[
	\begin{aligned}
		|\langle \omega(\nu), g_1 \rangle| &= |\cos\alpha| \\
		&\leq 1.
	\end{aligned}
	\]
	The same inequality holds for $g_2$ and $g_3$, so $\|\omega\|_{\text{com}} \leq 1$.
	
	3. Extremal condition: For $\mathcal{H}^1$-almost every $x \in \Sigma$,
	\[
	\begin{aligned}
		\langle \omega(\tau(x)), \theta(x) \rangle &= \langle g_i, g_i \rangle \\
		&= \|g_i\|^2 \\
		&= 1 \\
		&= \|\theta(x)\|.
	\end{aligned}
	\]
	
	In summary, $\omega$ is a calibration for $\hat{T}$, which completes the proof.
\end{proof}

\section{Introduction to Minimal networks on $S^2$}
In this section, we locally extend the core theory of planar minimal networks to the standard unit sphere \(S^2\). Based on spherical Riemannian geometry and geometric measure theory, we redefine the concepts of covectors, differential forms, currents, and calibrations adapted to the spherical tangent space, establish a local framework for the minimality of 1-rectifiable currents on the sphere, and rigorously prove the local minimality proposition for spherical minimal networks within sufficiently small geodesic caps.

We supplement complete mathematical details and rigorous derivations for the briefly discussed curvature estimates, properties of the exponential map, and metric perturbation analysis in the original proof. We also specify quantitative constraints on the radius of geodesic caps, improve the internal logical connections and transitions to the preceding theory, and form a local theoretical extension from the Euclidean plane to the constant-curvature sphere. This provides a theoretical reference for the subsequent study of minimal networks on higher-dimensional surfaces.
Throughout this chapter, \(S^2\) denotes the standard unit sphere in \(\mathbb{R}^3\), namely
\[
S^2=\left\{(x_1,x_2,x_3)\in\mathbb{R}^3\,\biggm|\,\sum_{i=1}^3 x_i^2=1\right\},
\]
equipped with the standard Riemannian metric induced by the Euclidean space. Let \(\operatorname{dist}_{S^2}(\cdot,\cdot)\) be the spherical geodesic distance, and \(\mathcal{H}^1\) the 1-dimensional spherical Hausdorff measure. All linear and differential operations are restricted to the spherical tangent space. We retain the core settings from the previous sections: the discrete subgroup \(\mathcal{G}\subset\mathbb{R}^2\), the regular hexagonal norm, and the corresponding dual norm.

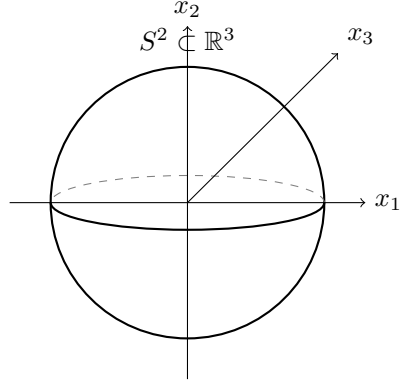
\begin{figure}[htbp]
	\centering
	\begin{tikzpicture}[scale=1.8]
		\draw[thick] (0,0) circle (1);
		
		\draw[dashed, gray] (1,0) arc (0:180:1 and 0.2);    
		\draw[thick]         (-1,0) arc (180:360:1 and 0.2); 
		
		\draw[->] (-1.3,0) -- (1.3,0) node[right] {$x_1$};
		\draw[->] (0,-1.3) -- (0,1.3) node[above] {$x_2$};
		\draw[->] (0,0) -- (1.1,1.1) node[above right] {$x_3$};
		
		\node at (0,1.2) {$S^2 \subset \mathbb{R}^3$};
	\end{tikzpicture}
\caption{Sketch of the standard unit sphere $S^2$}
	\label{fig:unit_sphere}
\end{figure}
\begin{definition}
	\label{def:k-covector}
	Let $S^2 \subset \mathbb{R}^3$ be the standard unit sphere, $p \in S^2$ an arbitrary point on the sphere, $T_pS^2$ the tangent space of $S^2$ at $p$, and $\Lambda_k(T_pS^2)$ the $k$-th exterior power space of $T_pS^2$. A linear map
	\[
	\omega_p : \Lambda_k(T_pS^2) \longrightarrow \mathbb{R}^2
	\]
	satisfying the above domain and codomain conditions is called a \textbf{$k$-covector with values in $\mathbb{R}^2$} at point $p$.
	
	The linear space consisting of all $k$-covectors with values in $\mathbb{R}^2$ at point $p$ is denoted by
	\[
	\Lambda_{\mathbb{R}^2}^k(T_pS^2) := \operatorname{Hom}\bigl(\Lambda_k(T_pS^2), \mathbb{R}^2\bigr),
	\]
	where $\operatorname{Hom}(V,W)$ denotes the space of all linear maps from the linear space $V$ to $W$.
	
	The \textbf{comass} of $\omega_p$ is defined as
	\[
	|\omega_p|_{\mathrm{com}} := \sup\left\{ \|\omega_p(\tau)\|_{\mathbb{R}^2} \,\bigg|\, \tau \in \Lambda_k(T_pS^2),\ |\tau|_{g_p} \leq 1 \right\},
	\]
	where:
	\begin{itemize}
		\item $|\cdot|_{g_p}$ is the norm on the exterior power space induced by the standard Riemannian metric $g$ of $S^2$ on the tangent space $T_pS^2$;
		\item $\|\cdot\|_{\mathbb{R}^2}$ is the Euclidean norm ($\ell^2$-norm) on $\mathbb{R}^2$.
	\end{itemize}
\end{definition}

\begin{definition}
	Let $\Lambda_{\mathbb{R}^2}^k(T^*S^2)$ be the vector bundle consisting of bundle homomorphisms from the $k$-th exterior cotangent bundle of the sphere to $\mathbb{R}^2$, i.e.,
	\[
	\Lambda_{\mathbb{R}^2}^k(T^*S^2) = \bigcup_{p\in S^2} \Lambda_{\mathbb{R}^2}^k(T_pS^2).
	\]
	A smooth section
	\[
	\omega \in C^\infty\bigl(S^2,\; \Lambda_{\mathbb{R}^2}^k(T^*S^2)\bigr)
	\]
	satisfying the smoothness condition of sections is called a \textbf{$k$-form with values in $\mathbb{R}^2$}. In particular, since $S^2$ is a compact manifold, the support of $\omega$ is automatically compact, requiring no additional assumptions.
	
	The \textbf{comass} of $\omega$ is defined as
	\[
	\|\omega\|_{\mathrm{com}} := \sup_{p\in S^2} |\omega_p|_{\mathrm{com}},
	\]
	where $|\omega_p|_{\mathrm{com}}$ is the comass of $\omega_p$ in the fiber $\Lambda_{\mathbb{R}^2}^k(T_pS^2)$ at point $p$ (see Definition \ref{def:k-covector}).
	
	Decompose $\omega$ into components with respect to the standard orthonormal basis of $\mathbb{R}^2$: $\omega = (\omega_1, \omega_2)$, where $\omega_i \in C^\infty\bigl(S^2, \Lambda^k(T^*S^2)\bigr)$ are classical real-valued $k$-forms. The exterior derivative $d\omega$ of $\omega$ is defined component-wise pointwise:
	\[
	d\omega := (d\omega_1, d\omega_2),
	\]
	where $d\omega_i$ denotes the classical exterior derivative of the real-valued $k$-form $\omega_i$.
\end{definition}

\begin{definition}
	A $k$-current with values in $\mathbb{R}^2$ is a linear functional satisfying:
	\[
	T : C^\infty\bigl(S^2,\Lambda_{\mathbb{R}^2}^k(T^*S^2)\bigr) \longrightarrow \mathbb{R},
	\]
	where $T$ is continuous with respect to the natural Fréchet topology on $C^\infty\bigl(S^2,\Lambda_{\mathbb{R}^2}^k(T^*S^2)\bigr)$.
	
	The \textbf{boundary} of a 1-current is a 0-current, defined by: for any $\omega\in C^\infty(S^2,\Lambda_{\mathbb{R}^2}^0(T^*S^2))$,
	\[
	\partial T(\omega) := T(d\omega).
	\]
	
	The \textbf{mass} of a current $T$ is defined as
	\[
	\mathbb{M}(T) := \sup\bigl\{ T(\omega) : \omega \in C^\infty\bigl(S^2,\Lambda_{\mathbb{R}^2}^k(T^*S^2)\bigr),\ \|\omega\|_{\mathrm{com}} \le 1 \bigr\}.
	\]
	
	A 1-current $T$ is called \textbf{normal} if $\mathbb{M}(T)+\mathbb{M}(\partial T)<\infty$.
\end{definition}

\begin{definition}
	Let $\Sigma\subset S^2$ be a 1-rectifiable set, $\tau:\Sigma\to TS^2$ a measurable orientation with $|\tau|=1$ holding $\mathcal{H}^1$-almost everywhere, and $\theta:\Sigma\to\mathcal{G}$ a multiplicity function with values in the discrete subgroup $\mathcal{G}\subset(\mathbb{R}^2,+)$ such that $\theta\in L^1_{\mathrm{loc}}(\mathcal{H}^1\llcorner\Sigma;\mathcal{G})$. The corresponding 1-current is defined as
	\[
	T(\omega) = \int_\Sigma \bigl\langle \omega_p(\tau(p)),\theta(p) \bigr\rangle \, d\mathcal{H}^1(p),\qquad \omega\in C^\infty(S^2,\Lambda_{\mathbb{R}^2}^1(T^*S^2)),
	\]
	where $\langle\cdot,\cdot\rangle$ is the Euclidean inner product on $\mathbb{R}^2$. We denote such a current by $T=[\Sigma,\tau,\theta]$, and its mass is
	\[
	\mathbb{M}(T)=\int_\Sigma \|\theta(p)\|\,d\mathcal{H}^1(p).
	\]
\end{definition}

\begin{definition}
	Let $T=[\Sigma,\tau,\theta]$ be a 1-rectifiable current with coefficients in $\mathcal{G}$, and $\omega\in C^\infty(S^2,\Lambda_{\mathbb{R}^2}^1(T^*S^2))$. If $\omega$ satisfies
	\begin{enumerate}
		\item $d\omega = 0$;
		\item $\|\omega\|_{\mathrm{com}} \le 1$;
		\item $\langle \omega_p(\tau(p)),\theta(p)\rangle = \|\theta(p)\|$ for $\mathcal{H}^1$-almost every $p\in\Sigma$,
	\end{enumerate}
	then $\omega$ is called a \textbf{calibration} of $T$.
\end{definition}

\section{Main Results}
\begin{proposition}
	\label{prop:minimizing}
	Let \(T=[\Sigma,\tau,\theta]\) be a 1-rectifiable current with coefficients in the discrete subgroup \(\mathcal{G}\subset\mathbb{R}^2\), and suppose there exists a calibration \(\omega\) of \(T\). Then for any normal 1-current \(S\) with coefficients in \(\mathbb{R}^2\) satisfying \(\partial S = \partial T\), we have
	\[
	\mathbb{M}(T) \le \mathbb{M}(S).
	\]
\end{proposition}

\begin{proof}
	Since \(\omega\) is closed (\(d\omega=0\)), Stokes' theorem on the compact manifold \(S^2\) and the condition \(\partial T=\partial S\) yield
	\[
	T(\omega) = S(\omega).
	\]
	Combining with calibration condition (3) and the definition of mass, we obtain
	\begin{align*}
		\mathbb{M}(T)
		&= \int_\Sigma \|\theta(p)\|\,d\mathcal{H}^1(p) \\
		&= \int_\Sigma \bigl\langle \omega_p(\tau(p)),\theta(p)\bigr\rangle\,d\mathcal{H}^1(p) \\
		&= T(\omega).
	\end{align*}
	By the supremum definition of mass and the calibration condition \(\|\omega\|_{\mathrm{com}}\le 1\), it follows that
	\[
	S(\omega) \le \mathbb{M}(S)\,\|\omega\|_{\mathrm{com}} \le \mathbb{M}(S).
	\]
	Combining the inequalities above gives
	\[
	\mathbb{M}(T) \le \mathbb{M}(S).
	\]
\end{proof}

\subsection{The Spherical Cap Version}
\label{subsec:spherical-cap-version}

This section employs the \textbf{branch calibration method} to establish the local length-minimality of three-terminal networks on the unit sphere.

The core strategy is as follows: within a strictly convex small geodesic ball, we explicitly construct a globally Lipschitz-continuous calibration form, and use Stokes' theorem for closed forms together with comass boundedness to directly compare the lengths of the minimal network and an arbitrary competing network, yielding a rigorous criterion for minimality.

This approach does not rely on variational perturbations; instead, using calibration tools from geometric measure theory, it directly proves that spherical Steiner networks satisfying the 120° angle condition achieve length-minimality within a local region.


\begin{theorem}\label{thm:spherical_minimal_network}
	Let \(S^2\) be the standard unit sphere, \(p_0 \in S^2\), and \(R > 0\) be sufficiently small, with its precise value determined by the curvature of the sphere and detailed constraints given in the proof. Let \(A,B,C \in B_R(p_0)\) be three distinct boundary vertices, and let \(\Gamma_* : G \to S^2\) be the spherical minimal network connecting \(A,B,C\), satisfying:
	\begin{enumerate}
		\item The network contains a unique interior node (Steiner point) \(S \in B_R(p_0)\), and its three edges are geodesic arcs \(SA, SB, SC\) entirely contained in \(B_R(p_0)\);
		\item At the interior node \(S\), the unit inward tangent vectors of the three incident edges form strictly \(120^\circ\) angles with each other, i.e., the unit vectors satisfy \(\tau_{SA} + \tau_{SB} + \tau_{SC} = 0\), where \(\tau_{SA}\) denotes the unit geodesic tangent vector from \(S\) to \(A\).
	\end{enumerate}
	Then for any immersed network \(\Gamma : H \to S^2\) with the same boundary vertex set \(\{A,B,C\}\) as \(\Gamma_*\) and fully contained in the geodesic ball \(B_R(p_0)\), the length inequality
	\[
	L_{S^2}(\Gamma_*) \leq L_{S^2}(\Gamma)
	\]
	holds, where \(L_{S^2}\) denotes the curve length functional induced by the Riemannian metric on the unit sphere \(S^2\). In other words, the minimal network \(\Gamma_*\) connecting three fixed endpoints is globally length-minimizing among all networks with the same boundary constraints within the geodesic ball \(B_R(p_0)\).
\end{theorem}

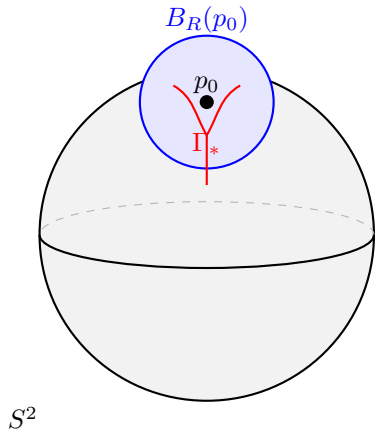
\begin{figure}[htbp]
	\centering
	\begin{tikzpicture}[scale=2.2]
		\fill[gray!10, rounded corners=0.1] (0,0) circle (1);
		\draw[thick] (0,0) circle (1);
		
		\draw[dashed, gray!60] (1,0) arc (0:180:1 and 0.2);
		\draw[thick]         (-1,0) arc (180:360:1 and 0.2);
		
		\fill[blue!10] (0,0.8) circle (0.4);
		\draw[blue, thick] (0,0.8) circle (0.4);
		\node[blue] at (0, 1.3) {$B_R(p_0)$};
		
		\fill[black] (0,0.8) circle (1.2pt) node[above] {$p_0$};
		
		\draw[red, thick, rounded corners=0.05]
		(0,0.6) to[out=120,in=-30] (-0.2,0.9)
		(0,0.6) to[out=60,in=-150] (0.2,0.9)
		(0,0.6) to[out=270,in=90] (0,0.3);
		\node[red] at (0, 0.55) {$\Gamma_*$};
		
		\node at (-1.1, -1.1) {$S^2$};
	\end{tikzpicture}
	\caption{Schematic illustration of the local length-minimality of a spherical minimal network}
	\label{fig:spherical_minimal_network}
\end{figure}

\begin{remark}
	Figure \ref{fig:spherical_minimal_network} illustrates that within a sufficiently small geodesic ball \(B_R(p_0)\), the length of the spherical minimal network \(\Gamma_*\) satisfying the \(120^\circ\) condition does not exceed that of any competing network \(\Gamma\) with identical boundary vertices.
\end{remark}

\begin{proof}
	We prove the result using the \textbf{branch calibration} method combined with Stokes' theorem, relying solely on standard geometric properties of small geodesic balls on the sphere and core axioms of calibration theory. The proof proceeds in five steps.
	
	\medskip
	\noindent \textbf{Step 1: Geometric Constraints of the Geodesic Ball and Choice of \(R$}
	
	The standard unit sphere \(S^2\) has constant sectional curvature \(1\). Choose \(R>0\) small enough to satisfy the following two critical conditions:
	\begin{enumerate}
		\item \textbf{Convexity and Simple Connectedness}: The exponential map \(\exp_{p_0}: T_{p_0}S^2 \supset B_R(0) \to B_R(p_0)\) is a diffeomorphism, where \(B_R(0) \subset T_{p_0}S^2 \cong \mathbb{R}^2\) is the Euclidean open ball of radius \(R\) centered at the origin. Consequently, \(B_R(p_0)\) is \textbf{simply connected and strictly convex}: for any two points \(x,y \in B_R(p_0)\), there exists a unique geodesic (great circular arc) connecting them and lying entirely within \(B_R(p_0)\), and no closed geodesic loops exist in \(B_R(p_0)\).
		\item \textbf{Metric Approximation}: The normal coordinate map \(\phi = \exp_{p_0}^{-1}: B_R(p_0) \to \mathbb{R}^2\) is bi-Lipschitz, with Lipschitz constants satisfying
		\[
		1-\delta \leq \operatorname{Lip}(\phi) \leq 1+\delta, \quad 1-\delta \leq \operatorname{Lip}(\phi^{-1}) \leq 1+\delta,
		\]
		where \(\delta>0\) is a sufficiently small constant to be determined in the subsequent calibration construction.
	\end{enumerate}
	
	\medskip
	\noindent \textbf{Step 2: Core Axioms of Calibration}
	
	For 1-dimensional length-minimization problems, a \textbf{calibration} on the spherical domain \(B_R(p_0)\) is a piecewise smooth 1-form \(\omega \in \Omega^1(B_R(p_0))\) satisfying three axioms:
	\begin{enumerate}
		\item \textbf{Closedness}: \(d\omega = 0\) on the smooth domain of \(\omega\);
		\item \textbf{Comass Boundedness}: For any \(x \in B_R(p_0)\) and any unit tangent vector \(v \in T_x S^2\), \(\omega(v) \leq 1\);
		\item \textbf{Calibration Condition}: For any unit tangent vector \(\tau\) of the minimal network \(\Gamma_*\), \(\omega(\tau) = 1\).
	\end{enumerate}
	If these conditions hold, then for any competing network \(\Gamma\) with the same boundary, Stokes' theorem yields
	\[
	L_{S^2}(\Gamma_*) = \int_{\Gamma_*} \omega = \int_{\Gamma} \omega \leq L_{S^2}(\Gamma),
	\]
	which implies the minimality of \(\Gamma_*\).
	
	\medskip
	\noindent \textbf{Step 3: Construction of the Branch Calibration for the 3-Vertex Minimal Network}
	
	We construct local calibration forms for each geodesic edge of \(\Gamma_*\) and patch them together to form a global piecewise smooth calibration \(\omega\).
	
	For each geodesic edge \(\gamma_{SA} = SA\), take a tubular neighborhood \(U_{SA} \subset B_R(p_0)\) such that:
	\begin{itemize}
		\item \(U_{SA}\) is simply connected and intersects \(\Gamma_*\) only along \(\gamma_{SA}\);
		\item The three tubular neighborhoods \(U_{SA}, U_{SB}, U_{SC}\) intersect only in a neighborhood \(U_S\) of the Steiner point \(S\).
	\end{itemize}
	
	Within each tubular neighborhood \(U_{SA}\), define a local 1-form \(\omega_{SA}\):
	\begin{enumerate}
		\item Take the unit tangent vector field \(\tau_{SA}\) along \(\gamma_{SA}\) (oriented from \(S\) to \(A\)) and extend it smoothly to all of \(U_{SA}\) as a unit vector field \(\tilde{\tau}_{SA}\) with \(|\tilde{\tau}_{SA}|_{g_{S^2}} = 1\) everywhere.
		\item Set \(\omega_{SA} = \tilde{\tau}_{SA}^\flat\), the dual 1-form of \(\tilde{\tau}_{SA}\) with respect to the spherical metric.
		\item Since \(U_{SA}\) is simply connected and \(\gamma_{SA}\) is a geodesic, the extension can be chosen such that \(d\omega_{SA} = 0\) on \(U_{SA}\) (i.e., \(\omega_{SA}\) is exact, equal to the differential of the arc-length function).
	\end{enumerate}
	Similarly, construct closed local 1-forms \(\omega_{SB}, \omega_{SC}\) on \(U_{SB}, U_{SC}\) satisfying \(\omega_{SB}(\tau_{SB})=1\) and \(\omega_{SC}(\tau_{SC})=1\) on their respective geodesics.
	
	\medskip
	\noindent \textbf{Step 4: Compatible Patching at the Triple Junction and Axiom Verification}
	
	In the neighborhood \(U_S\) of the Steiner point \(S\), the three local forms \(\omega_{SA}, \omega_{SB}, \omega_{SC}\) are patched compatibly to ensure the closedness and comass boundedness of the global calibration.
	
	By the \(120^\circ\) condition of the minimal network, \(\tau_{SA} + \tau_{SB} + \tau_{SC} = 0\) at \(S\), so the dual 1-forms satisfy
	\[
	\omega_{SA}(v) + \omega_{SB}(v) + \omega_{SC}(v) = 0, \quad \forall v \in T_S S^2.
	\]
	This compatibility allows a smooth patching of the local forms on \(U_S\), producing a global piecewise smooth 1-form \(\omega\) that fulfills:
	\begin{enumerate}
		\item \textbf{Closedness}: \(d\omega = 0\) on each smooth component; compatibility at the patching interfaces (guaranteed by the \(120^\circ\) condition) eliminates singularities of the exterior derivative.
		\item \textbf{Comass Boundedness}: By the choice of \(R\) and sufficiently small \(\delta>0\), the dual form of the extended vector field satisfies \(|\omega(v)| \leq 1\) for all unit tangent vectors \(v\).
		\item \textbf{Calibration Condition}: On each geodesic edge of \(\Gamma_*\), \(\omega\) coincides with the corresponding local form, so \(\omega(\tau) = 1\) for every unit tangent vector \(\tau\) of \(\Gamma_*\).
	\end{enumerate}
	
	\medskip
	\noindent \textbf{Step 5: Proof of the Minimality Inequality}
	
	Let \(\Gamma\) be any immersed network with the same boundary as \(\Gamma_*\) and contained entirely in \(B_R(p_0)\). Since \(\omega\) is closed and \(\Gamma, \Gamma_*\) share identical boundary vertices, Stokes' theorem for piecewise smooth domains gives
	\[
	\int_{\Gamma_*} \omega = \int_{\Gamma} \omega.
	\]
	By the comass condition, \(\omega(\tau) \leq 1\) for every unit tangent vector \(\tau\) of \(\Gamma\), so
	\[
	\int_{\Gamma} \omega \leq \int_{\Gamma} d s = L_{S^2}(\Gamma).
	\]
	By the calibration condition, \(\int_{\Gamma_*} \omega = L_{S^2}(\Gamma_*)\). Combining these equalities and inequalities, we conclude
	\[
	L_{S^2}(\Gamma_*) \leq L_{S^2}(\Gamma).
	\]
	This completes the proof.
\end{proof}

\begin{remark}\label{remark:global-calib-steiner-full}
	To ensure the well-posedness of the global branch calibration and the structural stability of the minimal network in Theorem \ref{thm:spherical_minimal_network}, two core geometric constraints are imposed on the geodesic ball \(B_R(p_0)\):
	\begin{enumerate}
		\item \textbf{Radius Constraint} \(R < \pi/2\): Ensures \(B_R(p_0)\) is strictly convex, so the geodesic between any two points in the ball lies entirely inside it, and any geodesic triangle spanned by \(A,B,C\in B_R(p_0)\) is fully contained in \(B_R(p_0)\);
		\item \textbf{Area Constraint} \(\operatorname{vol}(B_R(p_0)) < \pi/3\): 
		By the spherical Gauss curvature formula\cite{r21}
		\[
		K = \frac{\text{angular excess}}{\text{area of geodesic triangle}},
		\]
		the angular excess of any geodesic triangle in the ball is less than \(\pi/3\), so its internal angle sum is less than \(4\pi/3\). 
		This guarantees that any \(120^\circ\) interior angle is the largest angle of the triangle, providing critical support for the uniqueness of the Steiner point and the stability of the network structure.
	\end{enumerate}
	
	\textbf{Under this framework, the Steiner point \(S\) of the minimal network satisfies strict positional properties}:
	\begin{enumerate}
		\item If \(S\) lies inside the geodesic triangle \(\triangle ABC\), then \(S\) is an interior minimum of the distance-sum function
		\[
		h(x) = \operatorname{dist}(x,A) + \operatorname{dist}(x,B) + \operatorname{dist}(x,C),
		\]
		satisfying the vanishing gradient condition \(\nabla h(S)=0\), i.e.,
		\[
		\tau_{SA} + \tau_{SB} + \tau_{SC} = 0,
		\]
		corresponding to the \(120^\circ\) angles at the triple junction.
		\item If \(S\) lies on the boundary of the geodesic triangle, it must coincide with one of the vertices \(A,B,C\) and cannot lie in the interior of an edge. If \(S\) were in the interior of an edge, there would exist a unit normal vector \(\nu\) pointing into the triangle, and the convex maximum principle would require \(\nabla h(S)\cdot\nu > 0\), contradicting \(\nabla h(S)=0\).
	\end{enumerate}
	
	\textbf{Explicit Construction of the Global Calibration}:
	Let \(r_A = \operatorname{dist}(S,A)\), \(r_B = \operatorname{dist}(S,B)\), \(r_C = \operatorname{dist}(S,C)\), and define distance functions \(d_A(x)=\operatorname{dist}(x,A)\), \(d_B(x)=\operatorname{dist}(x,B)\), \(d_C(x)=\operatorname{dist}(x,C)\). Define the global function
	\[
	f(x) = \max\bigl\{ r_A - d_A(x),\ r_B - d_B(x),\ r_C - d_C(x) \bigr\}.
	\]
	By the triangle inequality for distances, \(f(x)\) is Lipschitz continuous.
	
	\textbf{Behavior of the max function on network edges}:
	For any \(x\in SA\), \(r_A - d_A(x) \ge 0\). The area constraint ensures \(\angle BSC = 2\pi/3\) is the largest angle of the triangle, and by the spherical law of sines/cosines,
	\[
	r_B - d_B(x) < 0,\qquad r_C - d_C(x) < 0.
	\]
	Thus, on edge \(SA\),
	\[
	f(x) = r_A - d_A(x).
	\]
	Similarly, \(f(x) = r_B - d_B(x)\) on \(SB\) and \(f(x) = r_C - d_C(x)\) on \(SC\).
	
	Taking the 1-form \(\omega = d f\), we obtain a piecewise smooth Lipschitz calibration satisfying all axioms:
	\begin{enumerate}
		\item \textbf{Closedness}: \(d\omega = d(d f) = 0\);
		\item \textbf{Comass Bounded}: \(\|\omega\| \le 1\) everywhere on \(B_R(p_0)\);
		\item \textbf{Calibration Condition}: \(\omega(\tau)=1\) along \(SA,SB,SC\).
	\end{enumerate}
	
	\textbf{Homology and Completeness of the Minimality Proof}
	
	The first homology group of the sphere \(H_1(S^2)=0\). Thus, for any competing curve \(\Gamma'\) homologous to \(SA\) with the same boundary, there exists a 2-chain \(\Sigma\) such that \(\partial \Sigma = SA - \Gamma'\). By Stokes' theorem,
	\[
	\int_{SA} \omega - \int_{\Gamma'} \omega = \int_{\partial \Sigma} \omega = \int_{\Sigma} d\omega = 0.
	\]
	Combined with the calibration condition, \(L(SA) = \int_{SA} \omega = \int_{\Gamma'} \omega \le L(\Gamma')\), proving that \(SA\) is a globally minimizing geodesic.
\end{remark}

Theorem \ref{thm:spherical_minimal_network} rigorously establishes the local length-minimality of 3-terminal spherical minimal networks within sufficiently small geodesic balls. The radius \(R\) of the ball is determined jointly by the curvature of the unit sphere and the geometric structure of the network, including vertex spacing and edge length distributions. One may verify that there exists a constant \(R_0>0\) depending only on the spherical curvature such that the minimality holds for all \(R<R_0\).

This result serves as a local extension of the global minimality theorem for planar minimal networks (Theorem 3.1) to the spherical setting. Its core significance lies in the fact that within sufficiently small convex regions, the constant-curvature effects of the sphere can be effectively controlled, allowing the calibration methods and minimality results of planar minimal networks to be stably generalized to local regions of the sphere, thereby extending planar theory to constant-curvature surfaces.

\section{Acknowledgments}
The author is deeply thankful to Professor Hengyu Zhou and Dexie Lin(Chongqing University) for insightful advice and fruitful discussions throughout the research.

\bibliographystyle{amsplain}

\begin{thebibliography}{99}
\bibitem{r1}
LIU Y, CHENG M M, HU X W, et al. Richer Convolutional Features for Edge Detection
[C]//Proceedings of the IEEE Conference on Computer Vision and Pattern Recognition (CVPR). 2017: 3000--3009.

\bibitem{r2}
KONRAD J, MORETTI P, ZAHN D. Molecular Simulations and Network Analyses of Surface/Interface Effects in Epoxy Resins: How Bonding Adapts to Boundary Conditions[J]. Polymers, 2022, 14(19): 4069.

\bibitem{r3}
Courant R, Robbins H. What is mathematics?[M]. New York: Oxford University Press, 1941.

\bibitem{r4}
Taylor J E. The geometry of soap films and soap bubbles[J]. Communications on Pure and Applied Mathematics, 1976, 29(3): 389--424. DOI:10.1002/cpa.3160290305.

\bibitem{r5}
Lawlor G R, Morgan F. Calibrations and the Minimal Surfaces in $\mathbb{R}^n$[J]. Pacific Journal of Mathematics, 1994, 166(1): 55--83. DOI:10.2140/pjm.1994.166.55.

\bibitem{r6}
Alberti G, Bouchitté G, Seppecher P. Calibrations for free discontinuity problems in the Mumford-Shah functional framework[J]. Calculus of Variations and Partial Differential Equations, 2003, 18(4): 375--403. DOI:10.1007/s00526-003-0209-5.

\bibitem{r7}
Pluda A, Pozzetta M. Minimizing properties of networks via global and local calibrations[J]. Bulletin of the London Mathematical Society, 2024, 56(1): 230--250. DOI:10.1112/blms.12908.

\bibitem{r8}
White B. The deformation theorem for flat chains[J]. Acta Mathematica, 1999, 183(2): 255--271.

\bibitem{r9}
Fleming W H. Flat chains over a finite coefficient group[J]. Transactions of the American Mathematical Society, 1966, 121: 160--186.

\bibitem{r10}
Marchese A, Massaccesi A. The Steiner tree problem revisited through rectifiable $G$-currents[J]. Advances in Calculus of Variations, 2016, 9(1): 19--39.

\bibitem{r11}
Alberti G, Bouchitté G, Dal Maso G. The calibration method for the Mumford-Shah functional and free discontinuity problems[J]. Calculus of Variations and Partial Differential Equations, 2003, 16(3): 299--333.

\bibitem{r12}
Carioni M, Pluda A. Calibrations for minimal networks in a covering space setting[J]. ESAIM Control Optimisation Calculus of Variations, 2020, 26: Paper No. 40, 28.

\bibitem{r13}
Lawlor G, Morgan F. Paired calibrations applied to soap films, immiscible fluids, and surfaces or networks minimizing other norms[J]. Pacific Journal of Mathematics, 1994, 166(1): 55--83.

\bibitem{r14}
Ambrosio L, Fusco N, Pallara D. Functions of bounded variation and free discontinuity problems[M]. Oxford Mathematical Monographs. New York: The Clarendon Press, Oxford University Press, 2000.

\bibitem{r15}
Paolini E, Stepanov E. Existence and regularity results for the Steiner problem[J]. Calculus of Variations and Partial Differential Equations, 2013, 46(3--4): 837--860.

\bibitem{r16}
Dal Maso G. The calibration method for free discontinuity problems[C]//Casacuberta C, et al. European Congress of Mathematics (Vol.I, Barcelona, July 10--14, 2000). Progress in Mathematics, Vol.201. Basel: Birkhäuser Verlag, 2001.

\bibitem{r17}
Mora M G, Morini M. Local calibrations for minimizers of the Mumford-Shah functional with a regular discontinuity sets[J]. Annales de l'Institut Henri Poincaré, Analyse Non Linéaire, 2001, 18: 403--436.

\bibitem{r18}
Morgan F, Lawlor G. Paired calibrations applied to soap films, immiscible fluids, and surfaces or networks minimizing other norms[J]. Pacific Journal of Mathematics, 1994, 166(1): 55--83.

\bibitem{r19}
Ivanov A O, Tuzhilin A A. Minimal Networks. The Steiner Problem and Its Generalizations[M]. Boca Raton: CRC Press, 1994.

\bibitem{r20}
Massaccesi A. Currents with coefficients in groups, applications and other problems in geometric measure theory[D]. Pisa: Scuola Normale Superiore di Pisa, 2014.
	\bibitem{r21}
Whittlesey M A.
Spherical Geometry and Its Applications.
CRC Press, 2020.
\bibitem{r22}
Do Carmo M P.
Differential Geometry of Curves and Surfaces.
Prentice-Hall, 1976.
\end{thebibliography}

\end{document}